\documentclass[final,leqno]{siamltex}

\usepackage{amsmath,amssymb,amscd,amsxtra,amsfonts,mathrsfs}
\usepackage{bm,epsf,graphicx,epsfig,color,latexsym,cite,cases}

\newtheorem{assumption}[theorem]{Assumption}
\newtheorem{remark}[theorem]{Remark}

\title{Inverse random source scattering for the Helmholtz equation with
attenuation\thanks{The research was supported in part by the NSF grant
DMS-1912704.}}

\author{Peijun Li\thanks{Department of Mathematics, Purdue University, West
Lafayette, Indiana 47907, USA. ({\tt lipeijun@math.purdue.edu}).}  \and Xu
Wang\thanks{Department of Mathematics, Purdue University, West Lafayette,
Indiana 47907, USA. ({\tt wang4191@purdue.edu}).} 
}

\begin{document}

\maketitle

\begin{abstract}
In this paper, a new model is proposed for the inverse random source scattering
problem of the Helmholtz equation with attenuation. The source is assumed to be
driven by a fractional Gaussian field whose covariance is represented by a
classical pseudo-differential operator. The work contains three contributions.
First, the connection is established between fractional Gaussian fields and
rough sources characterized by their principal symbols. Second, the direct
source scattering problem is shown to be well-posed in the distribution sense.
Third, we demonstrate that the micro-correlation strength of the random source
can be uniquely determined by the passive measurements of the wave field in a
set which is disjoint with the support of the strength function. The analysis
relies on careful studies on the Green function and Fourier integrals for the
Helmholtz equation.
\end{abstract}



\pagestyle{myheadings}
\thispagestyle{plain}
\markboth{P. Li and X. Wang}{Inverse Random Source Scattering}

\section{Introduction}

The inverse source scattering in waves is an important and active research
subject in inverse scattering theory. It is an important mathematical tool for
the solution of many medical imaging modalities \cite{ABF02, FKM04}. The inverse
source scattering problems are to determine the unknown sources that generate
prescribed wave patterns. These problems have attracted much research. The
mathematical and numerical results can be found in \cite{BLZ19, BLT10, IL18} and
the references cited therein. 

Stochastic modeling is widely introduced to mathematical systems due
to unpredictability of the environments, incomplete knowledge of the systems and
measurements, and fine-scale fluctuations in simulation. In many situations,
the source, hence the wave field, may not be deterministic but are rather
modeled by random processes \cite{D79}. Due
to the extra challenge of randomness and uncertainties, little is known for the
inverse random source scattering problems. 

In this paper, we consider the Helmholtz equation with a random source
\begin{equation}\label{eq:model}
\Delta u+(k^2+{\rm i} k\sigma)u=f,\quad x\in{\mathbb{R}}^d,
\end{equation}
where $d=2$ or $3$, $k>0$ is the wavenumber, the attenuation
coefficient $\sigma\ge0$ describes the electrical
conductivity of the medium, $u$ denotes the wave field, and $f$ is a random
function representing the electric current density. 

In \cite{BCL16}, the white noise model was studied for the inverse random
source problem of the stochastic Helmholtz equation without attenuation
\[
 \Delta u+k^2 u=g+h\dot{W},\quad x\in\mathbb R^d,
\]
where $g$ and $h$ are deterministic and compactly supported functions, and $\dot
W$ is the spatial white noise. It was shown that $g$ and $h$ can be determined
by statistics of the wave fields at multiple frequencies. The white noise model
can also be found in \cite{BCLZ14} and \cite{BCL17} for the one-dimensional
problem and the stochastic elastic wave equation, respectively. Recently, the
model of a generalized Gaussian field was developed to handle random processes
\cite{CHL19, HLO14}. The random function is said to be microlocally isotropic of
order $2s$ if the covariance operator is a pseudo-differential operator with
principal symbol given by $\mu(x)|\xi|^{-2s}$, where $\mu\geq 0$ is a smooth and
compactly support function and is called the micro-correlation strength of the
random function. It was shown that $\mu$ can be uniquely determined by the wave
field averaged over the frequency band at a single realization of the random
function. This model was also investigated in \cite{LHL, LL19} for the inverse
random source problems of the elastic wave equation and the Helmholtz equation
without attenuation. In these work, the parameter $s\in[\frac
d2,\frac{d}2+1)$ and the random functions are smoother than the
white noise (cf. Lemma \ref{lm:regu}): it can be interpreted as a distribution
in $W^{-\epsilon,p}({\mathbb{R}}^d)$ for any $\epsilon>0$ and $p\in(1,\infty)$
if $s=\frac d2$; it is a function in $C^{0, \alpha}({\mathbb{R}}^d)$
for any $\alpha\in(0,s-\frac d2)$ if $s\in(\frac d2,\frac d2+1)$.

In this work, we consider a new model for the Helmholtz equation
\eqref{eq:model}, where the random source $f$ is driven by a fractional
Gaussian field with $s\in[0,\frac d2+1)$. There are three contributions. First,
we demonstrate that the fractional Gaussian fields include the classical
fractional Brownian fields. Moreover, we establish the connection between the
fractional Gaussian fields and rough sources characterized by their principal
symbols. Second, we examine the regularity of the random source and show that
the direct scattering problem is well-posed in the distribution sense. Third,
for the inverse problem, we prove that the strength of the random source $\mu$
can be uniquely determined by the high frequency limit of the second moment of
the wave field. In particular, if $\sigma=0$, the strength $\mu$ can also be
determined uniquely by the amplitude of the wave field averaged over the
frequency band at a single realization of the random source. It is worthy to be
pointed out that (1) if $s\in[0,\frac d2]$, the random function is a
distribution in $f\in W^{s-\frac d2-\epsilon,p}$ for any $\epsilon>0$ and
$p\in(1,\infty)$ (cf. Lemma \ref{lm:regu}), which is rougher than those
considered in \cite{CHL19, HLO14, LHL, LL19}; (2) if $\sigma=0$ and $s\in[\frac
d2,\frac d2+1)$, the results obtained in this paper coincides with the ones
given in \cite{LHL}.

The paper is organized as follows. In Section \ref{sec:source}, the random
source model is introduced. The relationship is established between
the fractional Gaussian field and the classical fractional Brownian
motion; the regularity is studied for the random source. Section
\ref{sec:direct} addresses the well-posedness and regularity of the solution for
the direct problem. The inverse problem is discussed in Section
\ref{sec:inverse}, where the two- and three-dimensional problems are considered
separately. The paper is concluded with some general remarks and directions for
future work in Section \ref{sec:conc}.

\section{Random source}
\label{sec:source}

In this section, we give a general description of the random source on
${\mathbb{R}}^d$. Let $f$ be a real-valued centered random
field defined on a completed probability space
$(\Omega,\mathcal{F},\mathbb{P})$. Introduce the following Sobolev spaces. The
details can be found in \cite{AF03}. 

\begin{itemize}

\item $W^{s,p}:=W^{s,p}({\mathbb{R}}^d)$ for $s\in{\mathbb{R}}$ and
$p\in(1,\infty)$. 
In particular, if $p=2$, denote $H^{s}:=W^{s,2}$.

\item Denote by $W_{\rm loc}^{s,p}$ the space of functions which are locally in
$W^{s,p}$. More precisely, for any precompact subset
${\mathcal{O}}\subset{\mathbb{R}}^d$, $u|_{{\mathcal{O}}}\in
W^{s,p}({\mathcal{O}})$. 

\item Denote by $W_{\rm comp}^{s,p}$ the space of functions in $W^{s,p}$ with
compact support. 

\item Denote by $W_0^{s,p}({\mathcal{O}})$ the closure of
$C_0^{\infty}({\mathcal{O}})$ in $W^{s,p}({\mathcal{O}})$ with
${\mathcal{O}}\subset{\mathbb{R}}^d$. In particular, if
${\mathcal{O}}={\mathbb{R}}^d$, $W_0^{s,p}=W^{s,p}$.

\end{itemize}

Let $f:\Omega\to\mathcal{S}'$ be measurable such that the mapping
$\omega\mapsto\langle f(\omega),\phi\rangle$ defines a Gaussian random variable
for any $\phi\in C_0^{\infty}$. Here, $\mathcal{S}'$ is the space of
distributions on ${\mathbb{R}}^d$, which is the dual space of the Schwartz space
$\mathcal{S}$. The covariance operator $Q_f:C_0^{\infty}\to\mathcal{S}'$ is
given by
\[
\langle \varphi,Q_f\psi\rangle={\mathbb{E}}[\langle f,\varphi\rangle\langle
f,\psi\rangle]\quad\forall\,\varphi,\psi\in C_0^{\infty},
\]
where $\langle\cdot,\cdot\rangle$ denotes the dual product.
Denote by $K_f(x,y)$ the Schwartz kernel of $Q_f$, which satisfies
\[
\langle \varphi,Q_f\psi\rangle=\int_{{\mathbb{R}}^d}\int_{{\mathbb{R}}^d}
K_f(x,y)\varphi(x)\psi(y)dxdy. 
\]
Hence we have the following formal expression of the Schwartz kernel:
\[
K_f(x,y)={\mathbb{E}}[f(x)f(y)].
\]

\begin{assumption}\label{as:f}
The source $f$ is assumed to have a compact support contained
in ${\mathcal{D}}\subset{\mathbb{R}}^d$.
The covariance operator $Q_f$ of $f$ is a classical pseudo-differential operator
with the principal symbol $\mu(x)|\xi|^{-2s}$, where $s\in\left[0,\frac
d2+1\right)$ and $0\leq \mu\in C_0^{\infty}({\mathcal{D}})$. 
\end{assumption}

The positive function $\mu$ stands for the micro-correlation strength of the
random field $f$. The assumption implies that the covariance operator $Q_f$
satisfies
\begin{align*}
(Q_f\psi)(x)=\frac1{(2\pi)^d}\int_{{\mathbb{R}}^d}e^{{\rm i} x\cdot
\xi}c(x,\xi)\hat\psi(\xi)d\xi\quad\forall\,\psi\in C_0^{\infty},
\end{align*}
where the symbol $c(x,\xi)$ has the leading term $\mu(x)|\xi|^{-2s}$ and 
\[
\hat\psi(\xi)=\mathcal{F}[\psi](\xi)=\int_{{\mathbb{R}}^d}e^{-{\rm i}
x\cdot\xi}\psi(x)dx
\]
is the Fourier transform of $\psi$ \cite{H03,H07}. By the expression of
$Q_f\psi$, we can deduce the relationship between the kernel $K_f$
and the symbol $c(x,\xi)$. In fact, noting that
\begin{align*}
\langle\varphi,Q_f\psi\rangle
=&\int_{{\mathbb{R}}^d}\varphi(x)\left[\frac1{(2\pi)^{d}}\int_{{\mathbb{R}}^d}e^
{{\rm i} x\cdot \xi}c(x,\xi)\hat\psi(\xi)d\xi\right]dx\\
=&\frac1{(2\pi)^{d}}\int_{{\mathbb{R}}^d}\varphi(x)\int_{{\mathbb{R}}^d}e^{{
\mathbf{i}} x\cdot \xi}c(x,\xi)\left[\int_{{\mathbb{R}}^d}e^{-{\rm i}
y\cdot\xi}\psi(y)dy\right]d\xi dx\\
=&\int_{{\mathbb{R}}^d}\int_{{\mathbb{R}}^d}\left[\frac1{(2\pi)^d}\int_{{\mathbb
{R}}^d}e^{{\rm i}(x-y)\cdot\xi}c(x,\xi)d\xi\right]\varphi(x)\psi(y)dxdy,
\end{align*}
we get that the kernel $K_f$ is an oscillatory integral of the form
\begin{equation}\label{eq:Kf}
K_f(x,y)=\frac1{(2\pi)^d}\int_{{\mathbb{R}}^d}e^{{\rm i}(x-y)\cdot\xi}c(x,
\xi)d\xi.
\end{equation}

\subsection{Fractional Gaussian fields}

We introduce the fractional Gaussian fields, which can be used to generate
random fields satisfying Assumption \ref{as:f}.

\begin{definition}\label{df:FGF}
The fractional Gaussian field $h^s$ on ${\mathbb{R}}^d$ with parameter
$s\in{\mathbb{R}}$ is given by 
\[
h^s:=(-\Delta)^{-\frac s2}\dot{W},
\]
where $(-\Delta)^{-\frac s2}$ is the fractional Laplacian on ${\mathbb{R}}^d$
defined by
\begin{equation}\label{eq:fraclap}
(-\Delta)^{\alpha}u=\mathcal{F}^{-1}\left[|\xi|^{2\alpha}\mathcal{F}[u]
(\xi)\right],\quad \alpha\in\mathbb{R},
\end{equation}
and $\dot{W}\in\mathcal{S}'$ is the white noise on ${\mathbb{R}}^d$
determined by the covariance operator $Q_{\dot{W}}:L^2\to L^2$ as follows:
\[
\langle\varphi,Q_{\dot{W}}\psi\rangle:={\mathbb{E}}[\langle\dot{W},
\varphi\rangle\langle\dot{W},\psi\rangle]=(\varphi,\psi)_{L^2}
\quad\forall\,\varphi , \psi\in L^2.
\]
\end{definition}

We denote by $\mathbb G_s({\mathbb{R}}^d)$ the space of fractional
Gaussian fields with parameter $s$. Let $h^s\sim
\mathbb G_s({\mathbb{R}}^d)$ if $h^s$ is a fractional Gaussian field on
${\mathbb{R}}^d$ with parameter $s$.
If $d=1$ and $s=1$, $h^1$ turns to be the classical one-dimensional Brownian
motion. If $s=0$, $h^0=\dot{W}$ is the white noise on ${\mathbb{R}}^d$.
If $s<0$, $h^s$ is even rougher than the white noise.
We refer to \cite{LSSW16} and references therein for more details about the 
fractional Gaussian fields and the fractional Laplacian. 

To make sense of the expression $h^s=(-\Delta)^{-\frac s2}\dot{W}$, we define
\[
\mathcal{S}_r:=
\begin{cases}
\left\{\varphi\in\mathcal{S}:\int_{{\mathbb{R}}^d}x^\alpha\varphi(x)dx=0
\quad \forall\,|\alpha|\le r\right\} &\quad\text{if}~r\ge0\\
\mathcal{S}&\quad\text{if}~r<0
\end{cases}
\]
Denote by $T_s$ the closure of $\mathcal{S}_{s-\frac d2}$ in $H^{-s}$. 
Then the expression $h^s=(-\Delta)^{-\frac s2}\dot{W}$ in Definition
\ref{df:FGF} is interpreted by
\[
\langle h^s,\varphi\rangle:=\langle\dot{W},(-\Delta)^{-\frac
s2}\varphi\rangle=\int_{{\mathbb{R}}^d}(-\Delta)^{-\frac
s2}\varphi(x)dW(x)\quad\forall\,\varphi\in T_s.
\]
The kernel $K_{h^s}$ for the covariance operator $Q_{h^s}$ of $h^s$
satisfies
\begin{equation}\label{eq:Qhs}
\langle \varphi,Q_{h^s}\psi\rangle=\int_{{\mathbb{R}}^d}\int_{{\mathbb{R}}^d}
K_{h^s}(x,y)\varphi(x)\psi(y)dxdy\quad\forall\,\varphi,\psi\in C_0^{\infty}\cap
T_s. 
\end{equation}
Moreover, the kernel has the following expression. The proof can be found in
\cite{LSSW16}. 

\begin{lemma}\label{lm:kernel}
Let $h^s\sim\mathbb G_s({\mathbb{R}}^d)$ with parameter $s\in[0,\infty)$. Denote
$H:=s-\frac d2$. 
\begin{itemize}
\item[(i)] If $s\in(0,\infty)$ and $H$ is not a nonnegative integer, then
\[
K_{h^s}(x,y)=C_1(s,d)|x-y|^{2H},
\]
where  $C_1(s,d)=2^{-2s}\pi^{-\frac
d2}\Gamma(\frac d2-s)/\Gamma(s)$ with $\Gamma(\cdot)$ being the Gamma function.
\item[(ii)] If $s\in(0,\infty)$ and $H$ is a nonnegative integer, then
\[
K_{h^s}(x,y)=C_2(s,d)|x-y|^{2H}\ln|x-y|,
\]
where 
$C_2(s,d)=(-1)^{H+1}2^{-2s+1}\pi^{-\frac d2}/(H!\Gamma(s)).$
\item[(iii)] If $s=0$, then 
\[
K_{h^s}(x,y)=\delta(x-y),
\]
where $\delta(\cdot)$ is the Dirac delta function centered at $0$.
\end{itemize}
\end{lemma}

\subsection{Relationship with classical fractional Brownian fields}

For any $h^s\sim \mathbb G_s({\mathbb{R}}^d)$, we define its generalized Hurst
parameter $H=s-\frac d2$. If $s\in(\frac d2,\frac d2+1)$,
$h^s$ coincides with the classical fractional Brownian fields $B^H$ determined
by the covariance operator $Q_{B^H}$:
\begin{equation}\label{eq:QBH}
\langle \varphi,Q_{B^H}\psi\rangle=\int_{{\mathbb{R}}^d}\int_{{\mathbb{R}}^d}
\frac12\left[|x|^{2H}+|y|^{2H}-|x-y|^{2H}\right]\varphi(x)\psi(y)dxdy,
\end{equation}
where the Hurst parameter $H\in(0,1)$. 

\begin{lemma}\label{lm:fBm}
Let $s\in(\frac d2,\frac d2+1)$ and $h^s\sim\mathbb G_s({\mathbb{R}}^d)$. Then
the stochastic process defined by 
\[
\tilde{h}^s(x)=\langle h^s,\delta_x-\delta_0\rangle
\] 
has the same distribution as the fractional Brownian field $B^H$ with 
$H=s-\frac d2\in(0,1)$ up to a multiplicative constant, where
$\delta_x(\cdot)\in H^{-s}$ is the Dirac measure centered at
$x\in{\mathbb{R}}^d$.
\end{lemma}

\begin{proof}
By Theorem \ref{lm:kernel}, the kernel of the covariance operator reads
\begin{eqnarray*}
{\mathbb{E}}[\tilde h^s(x)\tilde h^s(y)]
&=&{\mathbb{E}}[\langle h^s,\delta_x-\delta_0\rangle\langle
h^s,\delta_y-\delta_0\rangle]\\
&=&C_1(s,d)\int_{{\mathbb{R}}^d}\int_{{\mathbb{R}}^d}|r_1-r_2|^{2H}
(\delta_x-\delta_0)(r_1)(\delta_y-\delta_0)(r_2)dr_1dr_2\\
&=&C_1(s,d)\left(|x-y|^{2H}-|x|^{2H}-|y|^{2H}\right),
\end{eqnarray*}
which is a scalar multiple of the kernel of the covariance operator $Q_{B^H}$
defined in \eqref{eq:QBH}. The result then follows from the fact that the
distribution of a centered Gaussian random field is unique determined by its
covariance operator.
\end{proof}

Note that $\langle h^s,\delta_x-\delta_0\rangle$ is actually a translation of
$h^s$. It indicates that we can identify $h^s\sim\mathbb G_s({\mathbb{R}}^d)$ as
the fractional Brownian field $B^H$ with $H=s-\frac d2\in(0,1)$ by fixing its
value to be zero at the origin. 
Define a random function 
\begin{equation}\label{eq:f}
f(x,\omega):=a(x)h^s(x,\omega),\quad x\in{\mathbb{R}}^d,~\omega\in\Omega,
\end{equation}
where $s\in(\frac d2,\frac d2+1)$ and $a\in C_0^{\infty}$ with
supp$(a)\subset{\mathcal{D}}$. We claim that such an $f$ defined above satisfies
Assumption \ref{as:f}. More precisely, the covariance operator $Q_f$ of $f$ has
the principal symbol $a^2(x)|\xi|^{-2s}$ up to a multiplicative constant.

\begin{proposition}\label{ps:rf}
The random field $f$ defined in \eqref{eq:f} with $s\in[0,\frac d2+1)$ satisfies
Assumption \ref{as:f} with $\mu=a^2$.
\end{proposition}

\begin{proof}
According to the expression of the kernel \eqref{eq:Kf} and Definition
\ref{df:FGF}, the covariance operator $Q_{h^s}$ of $h^s$ satisfies
\begin{eqnarray*}
\langle\varphi,Q_{h^s}\psi\rangle&=&~\mathbb{E}\left[\langle
h^s,\varphi\rangle\langle
h^s,\psi\rangle\right]=\mathbb{E}\left[\int_{{\mathbb{R}}^d}(-\Delta)^{-\frac
s2}(a\varphi) dW\int_{{\mathbb{R}}^d}(-\Delta)^{-\frac s2}(a\psi) dW\right]\\
&=&\int_{\mathbb{R}^d}(-\Delta)^{-\frac s2}(a\varphi)(-\Delta)^{-\frac
s2}(a\psi) dx\\
&=&~\frac1{(2\pi)^d}\int_{\mathbb{R}^d}\overline{\mathcal{F}\left[(-\Delta)^{
-\frac s2}(a\varphi)\right](\xi)}\mathcal{F}\left[(-\Delta)^{-\frac
s2}(a\psi)\right](\xi)d\xi,
\end{eqnarray*}
where the Plancherel theorem is used in the last step. It follows from the
definition of the fractional Laplacian given in \eqref{eq:fraclap} that we get
\begin{eqnarray*}
\langle\varphi,Q_{h^s}\psi\rangle&=&\frac1{(2\pi)^d}\int_{\mathbb{R}^d}|\xi|^{
-2s }\overline{\widehat{(a\varphi)}(\xi)}\widehat{(a\psi)}(\xi)d\xi\\
&=&~\frac1{(2\pi)^d}\int_{\mathbb{R}^d}|\xi|^{-2s}\left[\int_{\mathbb{R}^d}
a(x)\varphi(x)e^{{\rm i}x\cdot\xi}dx\right]\left[\int_{\mathbb{R}^d}
a(y)\psi(y)e^{-{\rm i}y\cdot\xi}dy\right]d\xi\\
&=&~\frac1{(2\pi)^d}\int_{\mathbb{R}^d}\int_{\mathbb{R}^d}\int_{\mathbb{R}^d}
\varphi(x)\psi(y)e^{{\rm i}(x-y)\cdot\xi}a^2(x)|\xi|^{-2s}d\xi dxdy\\
&&-\frac1{(2\pi)^d}\int_{\mathbb{R}^d}\int_{\mathbb{R}^d}\int_{\mathbb{R}^d}
\varphi(x)\psi(y)e^{{\rm i}(x-y)\cdot\xi}a(x)(a(x)-a(y))|\xi|^{-2s}d\xi
dxdy\\
&:=&I_1+I_2.
\end{eqnarray*}
Noting that $a(x)-a(y)=a'(\theta x+(1-\theta)y)(x-y)$ for some
$\theta\in(0,1)$ and 
\[
\int_{\mathbb{R}^d}e^{{\rm i}(x-y)\cdot\xi}|\xi|^{-2s}d\xi=(-\Delta)^{-s}
\delta(x-y),
\]
we obtain that the term $I_2$ is more regular than the term $I_1$. The proof is
completed by comparing the term $I_1$ with \eqref{eq:Kf}.
\end{proof}

\subsection{Regularity of random sources}

By Proposition \ref{ps:rf}, for any function $f$ satisfying Assumption
\ref{as:f} with parameter $s\in[0,\frac d2+1)$, its principal symbol has the
same order as the principal symbol of the random field
$a h^s$. 
Without loss of generality, we only need to investigate the regularity of random
fields given by $f=a h^s$, where $a\in C_0^{\infty}$ and
supp$(a)\subset{\mathcal{D}}$. 
Moreover, we assume that $f$ is a centered random field to avoid using the
modification $\langle h^s,\delta_x-\delta_0\rangle$.

\begin{lemma}\label{lm:regu}
Let $s\in[0,\frac d2+1)$ and $h\sim\mathbb G_s({\mathbb{R}}^d)$. Define the
random field $f:=a h^s$ with $a\in C_0^{\infty}$ and
supp$(a)\subset{\mathcal{D}}$.
\begin{itemize}
\item[(i)] If $s\in(\frac d2,\frac d2+1)$, then $f\in C^{0,\alpha}$ a.s. for
all $\alpha\in(0,s-\frac d2)$.
\item[(ii)] If $s\in[0,\frac d2]$, then $f\in W^{s-\frac d2-\epsilon,p}$ a.s.
for any $\epsilon>0$ and $p\in(1,\infty)$.
\end{itemize}
\end{lemma}

\begin{proof}
(i) If $s\in(\frac d2,\frac d2+1)$, it follows from Lemma \ref{lm:fBm} that $f$
has the same distribution as $aB^H$, where $B^H$ is a classical fractional
Brownian field on ${\mathbb{R}}^d$ with the Hurst parameter $H=s-\frac
d2\in(0,1)$. Note that
$B^H$ is $(H-\epsilon)$-H\"older continuous for any $\epsilon\in(0,H)$. Hence,
$f\in C^{0, \alpha}$ with $\alpha\in(0,H)=(0,s-\frac d2)$.

(ii) We first consider the case $s=\frac d2$ and hence $H=s-\frac d2=0$. By
Lemma \ref{lm:kernel}, the covariance operator $Q_f$ satisfies
\begin{eqnarray*}
\langle\varphi,Q_f\psi\rangle&=&\int_{{\mathbb{R}}^d}\int_{{\mathbb{R}}^d}K_f(x,
y)\varphi(x)\psi(y)dxdy\\
&=&\int_{{\mathbb{R}}^d}\int_{{\mathbb{R}}^d}a(x)a(y)K_{h^s}(x,
y)\varphi(x)\psi(y)dxdy\\
&=&\int_{{\mathbb{R}}^d}\int_{{\mathbb{R}}^d}C_2(s,
d)a(x)a(y)\ln|x-y|\varphi(x)\psi(y)dxdy
\end{eqnarray*}
for all $\varphi,\psi\in C_0^{\infty}\cap T_s$.
We may choose $K_f(x,y)=C_2(s,d)a(x)a(y)\ln|x-y|$ in this case. Following a 
similar proof to that of Theorem 2 in \cite{LPS08}, we consider the Bessel
potential operator $\mathcal{J}_{\epsilon}:=(I-\Delta)^{-\frac\epsilon2}$ with
$\epsilon>0$, where $I$ is the identify operator. It can be expressed through
the kernel in the form
$G_{\epsilon}(x,y)=C(\epsilon,d)|x-y|^{-d+\epsilon}+S(x,y)$ such that
\[
\mathcal{J}_{\epsilon}u=\int_{{\mathbb{R}}^d}G_{\epsilon}(x,y)u(y)dy,
\]
where $C(\epsilon,d)$ is a constant depending on $\epsilon$ and $d$, and
$S(x,y)$ is the more regular residual.
Note that $\mathcal{J}_\epsilon: W^{t,p}\to W^{t+\epsilon,p}$ is an isomorphism
for $t\in{\mathbb{R}}$ and $p\in(1,\infty)$ (see e.g. \cite[Section
3.3]{S70}). It then suffices to show that $\mathcal{J}_\epsilon f\in L^p$ a.s.
for any $\epsilon>0$ and $p\in(1,\infty)$. 
In fact, this result is obvious since the kernel of
$\mathcal{J}_{\epsilon}f$ is uniformly bounded:
\[
\int_{{\mathbb{R}}^d}\int_{{\mathbb{R}}^d}\left|a(x)a(y)\frac{\ln|x-y|}{|x-y|^{
d-\epsilon}}\right|dxdy\lesssim\int_0^R\left|\frac{\ln
r}{r^{d-\epsilon}}\right|r^{d-1}dr<\infty
\] 
for some positive number $R$ such that ${\mathcal{D}}\subset B(0,R)$.

If $s\in[0,\frac d2)$, the result can be obtained directly by noticing
$f=a(-\Delta)^{-\frac s2+\frac d4}\tilde f$ with $\tilde f:=(-\Delta)^{-\frac
d4}\dot{W}\in W^{-\epsilon,p}$ and the result obtained above for
$s=\frac{d}{2}$.
\end{proof}

\section{Direct scattering problem}
\label{sec:direct}

This section is to investigate the well-posedness and study the regularity of
the solution for the direct scattering problem. 

\subsection{Fundamental solution}

Let $\kappa^2=k^2+{\rm i} k\sigma$. A simple calculation yields that 
\[
 \Re[\kappa]=\kappa_{\rm
r}=\bigg(\frac{\sqrt{k^4+k^2\sigma^2}+k^2}2\bigg)^{\frac12 } , \quad
\Im[\kappa]=\kappa_{\rm
i}=\bigg(\frac{\sqrt{k^4+k^2\sigma^2}-k^2}2\bigg)^{\frac12},
\]
and
\begin{equation}\label{kri}
 \lim_{k\to\infty}\frac{\kappa_{\rm r}}{k}=1,\quad \lim_{k\to\infty}\kappa_{\rm
i}=\frac{\sigma}{2}. 
\end{equation}
Then the Helmholtz equation \eqref{eq:model} can be written as 
\begin{align}\label{eq:modelnew} 
\Delta u+\kappa^2u=f\quad\text{in}~\mathbb R^d.
\end{align}
The Helmholtz equation \eqref{eq:modelnew} with a complex-valued wavenumber
has the fundamental solution
\[
\Phi_\kappa(x,y)=
\begin{cases}
\frac{\rm i}4H_0^{(1)}(\kappa|x-y|), &\quad d=2,\\
\frac1{4\pi}\frac{e^{{\rm i}\kappa|x-y|}}{|x-y|}, &\quad d=3,
\end{cases}
\]
where $H_0^{(1)}$ is the Hankel function of the first kind with order 0.

\begin{lemma}\label{lm:Phik}
For any given $x\in{\mathbb{R}}^d$, it holds that 
$\Phi_\kappa(x,\cdot)\in W_{\rm loc}^{1,p},$ where $p\in(1,2)$ if $d=2$ and
$p\in(1,\frac32)$ if $d=3$.
\end{lemma}
\begin{proof}
Let $D\subset{\mathbb{R}}^d$ be any bounded domain. Denote $r^*:=\sup_{y\in
D}|x-y|$, then $D\subset B_{r^*}(x)$.

For $d=2$, we have $\Phi_\kappa(x,y)=\frac{\rm i}4H_0^{(1)}(\kappa|x-y|).$ It
suffices to show that $H_0^{(1)}(\kappa|x-\cdot|)\in L^{p}(D)$ and
$D_y^{\alpha}H_0^{(1)}(\kappa|x-\cdot|)\in L^p(D)$ with $0<|\alpha|\le1$.
Note that
\[
\left|H_\nu^{(1)}(z)\right|\le
e^{-\Im[z]\left(1-\frac{\Theta^2}{|z|^2}\right)^\frac12}\left|H_\nu^{(1)}
(\Theta)\right|
\]
for any $\nu\in{\mathbb{R}}$ and any real number $\Theta$ satisfying
$0<\Theta\le|z|$ (cf. \cite[Lemma 2.2]{CL05}). By choosing
$z=\kappa|x-y|$ and $\Theta=\Re(z)=\kappa_{\rm r}|x-y|$, on one hand, we have
\begin{align*}
\int_{D}\left|H_0^{(1)}(\kappa|x-y|)\right|^pdy\le&\int_De^{-p\frac{\kappa_{\rm
i}^2}{|\kappa|}|x-y|}\left|H_0^{(1)}(\kappa_{\rm r}|x-y|)\right|^pdy\\
\lesssim&\int_0^{r^*}e^{-p\frac{\kappa_{\rm i}^2}{|\kappa|}r}\left|H_0^{(1)}(
\kappa_{\rm r}r)\right|^prdr.
\end{align*}
For the above integral, since
$H_0^{(1)}(\kappa_{\rm r}r)\sim\frac{2{\rm i}}{\pi}\ln(\kappa_{\rm r}r)$
as $r\to0$ (cf. \cite[eq. (9.1.8)]{AS92}), we only need to consider the
integral
\[
\int_0^{r^*}e^{-p\frac{\kappa_{\rm i}^2}{|\kappa|}r}|\ln(\kappa_{\rm r}
r)|^prdr<\infty,
\]
which leads to $H_0^{(1)}(\kappa|x-\cdot|)\in L^{p}(D)$.
On the other hand, we have 
\[
\partial_{y_i}H_0^{(1)}(\kappa|x-y|)=\kappa
H_0^{(1)'}(\kappa|x-y|)\frac{y_i-x_i}{|x-y|}=-\kappa
H_1^{(1)}(\kappa|x-y|)\frac{y_i-x_i}{|x-y|}, \quad i=1,2. 
\]
Hence
\begin{align*}
\int_D\left|\partial_{y_i}H_0^{(1)}(\kappa|x-y|)\right|^pdy=&
|\kappa|^p\int_{D}\left|H_1^{(1)}(\kappa|x-y|)\right|^p\left|\frac{y_i-x_i}{
|x-y|}\right|^pdy\\
\lesssim&\int_De^{-p\frac{\kappa_{\rm i}^2}{|\kappa|}|x-y|}\left|H_1^{(1)}(
\kappa_{\rm r}|x-y|)\right|^pdy\\
\lesssim&\int_0^{r^*}e^{-p\frac{\kappa_{\rm
i}^2}{|\kappa|}r}\left|H_1^{(1)}(\kappa_{\rm r}r)\right|^prdr,
\end{align*}
where $H_1^{(1)}$ is the Hankel function of the first kind with order 1 and it
has the asymptotic expansion
$H_1^{(1)}(\kappa_{\rm r}r)\sim\frac{2{\rm i}}\pi\frac1{\kappa_{\rm r}r}$
as $r\to0$ (cf. \cite[eq. (9.1.9)]{AS92}). Since
\[
\int_0^{r^*}e^{-p\frac{\kappa_{\rm i}^2}{|\kappa|}r}\frac1{r^p}rdr<\infty,\quad
p\in (1, 2), 
\]
we obtain that $D^{\alpha}_yH_0^{(1)}(\kappa|x-\cdot|)\in L^p(D)$.

For $d=3$, we have 
$\Phi_{\kappa}(x,y)=\frac1{4\pi}\frac{e^{{\rm i}\kappa|x-y|}}{|x-y|}$ and
$\partial_{y_i}\Phi_\kappa(x,y)=\frac{e^{{\rm i}\kappa|x-y|}(y_i-x_i)}{
4\pi|x-y|^3}\left({\rm i}\kappa|x-y|-1\right)$, $i=1,2,3$. Noting for
$p\in(1,\frac32)$ that 
\[
\int_D\left|\frac{e^{{\rm i}\kappa|x-y|}}{|x-y|}\right|^pdy\lesssim\int_0^{
r^*}\frac1{r^p}r^2dr<\infty
\]
and
\[
\int_D\left|\frac{e^{{\rm i}\kappa|x-y|}(y_i-x_i)}{|x-y|^3}
\right|^pdy\lesssim\int_0^{r^*}\frac1{r^{2p}}r^2dr<\infty,
\]
we complete the proof.
\end{proof}

\subsection{Well-posedness and regularity}

Using the fundamental solution $\Phi_\kappa$, we define a volume potential
\begin{equation*}
(V_\kappa f)(x):=-\int_{\mathbb{R}^d}\Phi_\kappa(x,y)f(y)dy.
\end{equation*}
The mollifier $V_\kappa$ has the following property. The proof can be found in
\cite{LHL,LPS08}.  

\begin{lemma}\label{lm:bv}
Let $\mathcal{O}$ and $\mathcal{U}$ be two bounded domains in $\mathbb{R}^d$.
The operator $V_\kappa:H_0^{-\beta}({\mathcal{O}})\to
H^\beta({\mathcal{U}})$ is bounded for $\beta\in\left(0,2-\frac d2\right]$.
\end{lemma}

\begin{theorem}\label{tm:solution}
Let $p\in (\frac{d}2,2]$,
$s\in \big(d (\frac1p+\frac12 )-2,\frac d2\big]$, and $H=s-\frac d2\in (\frac
dp-2,0]$. Assume that $f\in W_{\rm comp}^{H-\epsilon,p}$ for
any $\epsilon>0$. Then the  scattering problem \eqref{eq:model}
admits a unique solution $u\in W_{\rm loc}^{-H+\epsilon,q}$ a.s. in the
distribution sense with $q$ satisfying
$\frac1p+\frac1q=1$. Moreover, the solution is given by 
\[
u(x;k)=-\int_{{\mathbb{R}^d}}\Phi_\kappa(x,y)f(y)dy.
\]
\end{theorem}

\begin{proof}
We only need to show the existence of the solution since the uniqueness follows
directly from the deterministic case. Let ${\mathcal{D}}$ be a bounded domain
such that supp$(f)\subset{\mathcal{D}}$.
Then $f\in W^{H-\epsilon,p}({\mathcal{D}})$. For any $x\in{\mathbb{R}}^d$,
define the volume potential 
\[
u_*(x;k):=-\int_{\mathcal{D}}\Phi_\kappa(x,y)f(y)dy=-\int_{{\mathbb{R}}^d}
\Phi_\kappa(x,y)f(y)dy.
\]

First we show that $u_*$ is a solution of \eqref{eq:model} in the distribution
sense. In fact, we have for any $v\in C_0^\infty$ that 
\begin{eqnarray*}
&&\langle\Delta u_*+\kappa^2u_*,v\rangle=-\langle\nabla u_*,\nabla
v\rangle+\kappa^2\langle u_*,v\rangle\\
&=&\int_{{\mathbb{R}}^d}\nabla_x \Big[\int_{{\mathbb{R}}^d}\Phi_\kappa(x,
y)f(y)dy \Big]\nabla
v(x)dx -\kappa^2\int_{{\mathbb{R}}^d}\Big[\int_{{\mathbb{R}}^d}\Phi_\kappa(x,
y)f(y)dy\Big] v(x)dx\\
&=&-\int_{{\mathbb{R}}^d}\int_{{\mathbb{R}}^d}\Delta_x\Phi_\kappa(x,
y)v(x)f(y)dxdy -\kappa^2\int_{{\mathbb{R}}^d}\Big[\int_{{\mathbb{R}}^d}
\Phi_\kappa(x,y)f(y)dy\Big] v(x)dx\\
&=&\int_{{\mathbb{R}}^d}\int_{{\mathbb{R}}^d}\left(\kappa^2\Phi_\kappa(x,
y)+\delta(x-y)\right)v(x)f(y)dxdy-\kappa^2\int_{{\mathbb{R}}^d}\Big[\int_{{
\mathbb{R}}^d}\Phi_\kappa(x,y)f(y)dy\Big] v(x)dx\\
&=&\langle f,v\rangle.
\end{eqnarray*}

It then suffices to show that $u_*\in W^{-H+\epsilon,q}_{\rm loc}$, which is
equivalent to show that $\phi u_*\in W^{-H+\epsilon,q}$ for any $\phi\in
C_0^{\infty}$ with support $\mathcal{U}\subset\mathbb{R}^d$.
Define a weighted potential
\[
(\tilde V_\kappa f)(x):=-\phi(x)\int_{{\mathbb{R}}^d}
\Phi_\kappa(x,y)f(y)dy,\quad x\in\mathcal{U}.
\]
By Lemma \ref{lm:bv}, the operator $\tilde
V_\kappa:H^{-\beta}_0({\mathcal{D}})\to H^{\beta}(\mathcal{U})$ is bounded
for $\beta\in\left(0,2-\frac d2\right]$. Noting the Sobolev embedding
theorem with fractional index that $W^{r,p}$ is embedded continuously into
$W^{t,q}$ with $r\ge
t$ and $\frac1q=\frac1p-\frac{r-t}{d}$, we get that
$W^{H-\epsilon,p}_0({\mathcal{D}})\hookrightarrow H^{-\beta}_0({\mathcal{D}})$
with $-H+\epsilon\le\beta$ and 
$
-H+\epsilon=d(\frac12-\frac1p)+\beta\in(0,2-\frac dp],
$ 
and $H^{\beta}(\mathcal{U})\hookrightarrow W^{-H+\epsilon,q}(\mathcal{U})$
with $\frac1p+\frac1q=1$. Consequently, $\tilde V_\kappa:
W^{H-\epsilon,p}_0({\mathcal{D}})\to
W^{-H+\epsilon,q}(\mathcal{U})$ is bounded, which shows that $
\phi u_*=\tilde V_\kappa f\in W^{-H+\epsilon,q}$ and completes the proof.
\end{proof}

\begin{remark}
It follows from Lemma \ref{lm:regu} that the random source is a continuous
function for $s\in(\frac d2,\frac d2+1)$. The well-posedness of the scattering
problem \eqref{eq:model} is well known since the source $f$ is compactly
supported and regular enough \cite{CK13}.
\end{remark}

\section{Inverse scattering problem}
\label{sec:inverse}

This section addresses the inverse scattering problem. The goal is to determine
the strength $\mu$ of the random source $f$. We discuss the two- and
three-dimensional cases, separately. 

\subsection{Two-dimensional case} 

First we consider $d=2$ in which $s\in[0,\frac
d2+1)=[0,2)$. Recall that the Hankel function has the following asymptotic
expansion \cite{AS92}:
\begin{equation}\label{eq:Hn}
H_0^{(1)}(z)\sim\sum_{j=0}^{\infty}a_{j}z^{-\left(j+\frac12\right)}e^{{\rm i
} z},\quad z\in\mathbb{C},~|z|\to\infty,
\end{equation}
where $a_{0}=\sqrt{\frac2\pi}e^{-\frac{{\rm i}\pi}4}$ and
$a_{j}=\sqrt{\frac2\pi}\left(\frac{{\rm i}}8\right)^j \Big(\prod\limits_{l=1
}^j(2l-1)^2/j!\Big)e^{-\frac{\rm i\pi}4}, j\ge1.$
Denoting 
\[
H_{0,N}^{(1)}(z):=\sum_{j=0}^Na_{j}z^{-\left(j+\frac12\right)}e^{{\rm i} z},
\quad 
\Phi_\kappa^N(x,y):=\frac{{\rm i}}4H_{0,N}^{(1)}(\kappa|x-y|),
\]
we have
\begin{align*}
\Phi_\kappa(x,y)=\Phi_\kappa^N(x,y)+O\big(|\kappa|x-y||^{
-(N+\frac32)}\big),\quad N\in\mathbb{N},
\end{align*}
as $|\kappa|x-y||\to\infty$ due to $\kappa_{\rm i}>0$.
Based on the truncated fundamental solution $\Phi_\kappa^2(x,y)$ by choosing
$N=2$, we consider the approximate solution 
\begin{eqnarray*}
&&u^2(x;k)=-\int_{{\mathbb{R}}^2}\Phi_\kappa^2(x,y)f(y)dy
=-\frac{{\rm i}
a_{0}}4\int_{{\mathbb{R}}^2}(\kappa|x-y|)^{-\frac12}e^{{\rm i}\kappa|x-y|}
f(y)dy\\
&&\quad -\frac{{\rm i}
a_{1}}4\int_{{\mathbb{R}}^2}(\kappa|x-y|)^{-\frac32}e^{{\rm i}\kappa|x-y|}
f(y)dy-\frac{{\rm i}
a_{2}}4\int_{{\mathbb{R}}^2}(\kappa|x-y|)^{-\frac52}e^{{\rm i}\kappa|x-y|}
f(y)dy,\quad x\in{\mathbb{R}}^2.
\end{eqnarray*}

Let $\mathcal{U}\subset\mathbb R^2$ be a bounded domain satisfying 
${\rm dist}({\mathcal{U}},{\mathcal{D}})=r_0>0$. First we show that the strength
$\mu$ of the source $f$ given in Assumption
\ref{as:f} can be reconstructed uniquely by the variance of the solution $u$ on
$\mathcal U$. 

\begin{proposition}\label{prop:u1}
Let $k\ge1$ and the assumptions in Theorem \ref{tm:solution} hold. Then the
following estimate holds:
\[
{\mathbb{E}}|u^2(x;k)|^2=T_\kappa(x)|\kappa|^{-1}\kappa_{\rm r}^{-2s}+O\left(
\kappa_{\rm r}^{-2s-2}\right),\quad x\in{\mathcal{U}},
\]
where
\[
T_{\kappa}(x):=\frac{1}{2^3\pi}\int_{{\mathbb{R}}^2}\frac{e^{-2\kappa_{\rm
i}|x-y|}}{|x-y|}\mu\left(y\right)dy.
\]
\end{proposition}

\begin{proof}
For any $x\in{\mathcal{U}}$, we have from straightforward calculations that
\begin{eqnarray}\label{eq:u1square}
{\mathbb{E}}|u^2(x;k)|^2&=&\frac{|a_{0}|^2}{16|\kappa|}\int_{{\mathbb{R}}^2}
\int_{{\mathbb{R}}^2}\frac{e^{{\rm i}\kappa|x-y|-{\rm i}\bar{\kappa}
|x-z|}}{|x-y|^{\frac12}|x-z|^{\frac12}}{\mathbb{E}}[f(y)f(z)]dydz\nonumber\\
&&+\Re\left[\frac{a_{0}\bar{a}_1}{8|\kappa|\bar{\kappa}}\right]
\int_{{\mathbb{R}}^2}\int_{{\mathbb{R}}^2}\frac{e^{{\rm
i}\kappa|x-y|-{\rm
i}\bar{\kappa}|x-z|}}{|x-y|^{\frac12}|x-z|^{\frac32}}{\mathbb{E}}[f(y)f(z)]
dydz\nonumber\\
&&+\frac{|a_{1}|^2}{16|\kappa|^3}\int_{{\mathbb{R}}^2}\int_{{\mathbb{R}}^2}\frac
{e^{{\rm i}\kappa|x-y|-{\rm i}\bar{\kappa}|x-z|}}{|x-y|^{\frac32}
|x-z|^{\frac32}}{\mathbb{E}}[f(y)f(z)]dydz\nonumber\\
&&+\Re\left[\frac{a_{0}\bar{a}_2}{8|\kappa|\bar{\kappa}^2}\right]
\int_{{\mathbb{R}}^2}\int_{{\mathbb{R}}^2}\frac{e^{{\rm
i}\kappa|x-y|-{\rm i}\bar{\kappa}|x-z|}}{|x-y|^{\frac12}|x-z|^{\frac52}}{
\mathbb{E}}[
f(y)f(z)]dydz\nonumber\\
&&+\Re\left[\frac{a_{1}\bar{a}_2}{8|\kappa|^3\bar{\kappa}}\right]
\int_{{\mathbb{R}}^2}\int_{{\mathbb{R}}^2}\frac{e^{{\rm
i}\kappa|x-y|-{\rm i}\bar{\kappa}|x-z|}}{|x-y|^{\frac32}|x-z|^{\frac52}}{
\mathbb{E}}[
f(y)f(z)]dydz\nonumber\\
&&+\frac{|a_{2}|^2}{16|\kappa|^5}\int_{{\mathbb{R}}^2}\int_{{\mathbb{R}}^2}\frac
{e^{{\rm i}\kappa|x-y|-{\rm i}\bar{\kappa}|x-z|}}{|x-y|^{\frac52}
|x-z|^{\frac52}}{\mathbb{E}}[f(y)f(z)]dydz. 
\end{eqnarray}
To estimate all the above terms, it suffices to consider the integral
\[
I_{l_1,l_2}(x;k):=\int_{{\mathbb{R}}^2}\int_{{\mathbb{R}}^2}\frac{e^{{\rm i}
\kappa|x-y|-{\rm i}\overline{\kappa}|x-z|}}{|x-y|^{\frac12+l_1}|x-z|^{
\frac12+l_2}}K_f(y,z)\theta(x)dydz,\quad l_1,l_2\in\{0,1,2\},
\]
where $\theta\in C_0^{\infty}$ such that $\theta|_{{\mathcal{U}}}\equiv1$ and
supp$(\theta)\subset{\mathbb{R}}^2\backslash\overline{{\mathcal{D}}}$.
Define $C_1(y,z,x):=K_f(y,z)\theta(x)$ and $c_1(y,x,\xi):=c(y,\xi)\theta(x)$
with $c(y,\xi)$ being the symbol of the covariance operator $Q_f$ of the random
field $f$. Furthermore, $c_1\in S^{-2s}$ with $S^m$ being the space of symbols
of order $m$, $m\in{\mathbb{R}}$, has the principal symbol
\[
c_1^p(y,x,\xi)=\mu(y)\theta(x)|\xi|^{-2s}.
\]
Based on \eqref{eq:Kf}, we have
\begin{align*}
C_1(y,z,x)=\frac1{(2\pi)^2}\int_{{\mathbb{R}}^2}e^{{\rm i}(y-z)\cdot\xi}
c_1(y,x,\xi)d\xi,
\end{align*}
which is compactly supported in
${\mathcal{D}}^{\theta}:={\mathcal{D}}\times{\mathcal{D}}\times\text{supp}
(\theta)$. Moreover, $C_1$ is a conormal distribution in ${\mathbb{R}}^6$ of
H\"ormander type having conormal singularity on the surface
$S:=\{(y,z,x)\in{\mathbb{R}}^6 : y-z=0\}$ and is invariant under a change of
coordinates \cite{H07}.

To calculate the integral in \eqref{eq:u1square}, different coordinates systems
will be considered. Define an invertible transformation
$\tau:{\mathbb{R}}^6\to{\mathbb{R}}^6$ by
\[
\tau(y,z,x)=(g,h,x),
\]
where $g=(g_1,g_2)$ and $h=(h_1,h_2)$ with 
\begin{align*}
g_1=\frac12 (|x-y|-|x-z|), \quad
g_2=\frac12\bigg[|x-y|\arcsin\Big(\frac{y_1-x_1}{|x-y|}
\Big)-|x-z|\arcsin\Big(\frac{z_1-x_1}{|x-z|}\Big)\bigg],\\
h_1=\frac12 (|x-y|+|x-z|),\quad
h_2=\frac12\bigg[|x-y|\arcsin\Big(\frac{y_1-x_1}{|x-y|}
\Big)+|x-z|\arcsin\Big(\frac{z_1-x_1}{|x-z|}\Big)\bigg].
\end{align*}
Under the new coordinates system, we get
\begin{eqnarray}\label{eq:u0square2}
I_{l_1,l_2}(x;k)&=&\int_{{\mathbb{R}}^2}\int_{{\mathbb{R}}^2}e^{{\rm i}
\kappa_{\rm
r}\left(|x-y|-|x-z|\right)-\kappa_{\rm i}\left(|x-y|+|x-z|\right)}\frac{
C_1(y,z,x)}{|x-y|^{\frac12+l_1}|x-z|^{\frac12+l_2}}dydz\nonumber\\
&=&\int_{{\mathbb{R}}^2}\int_{{\mathbb{R}}^2}e^{{\rm i}2\kappa_{\rm r}
\left(e_1\cdot g\right)-2\kappa_{\rm i}\left(e_1\cdot h\right)}C_2(g,h,x)dgdh,
\end{eqnarray}
where $e_1=(1,0)$ and 
\begin{align}\label{eq:C2}
C_2(g,h,x)=&~C_1(\tau^{-1}(g,h,x))\frac{\det((\tau^{-1})'(g,h,x))}{
\left((g+h)\cdot e_1\right)^{\frac12+l_1}\left((h-g)\cdot
e_1\right)^{\frac12+l_2}}\nonumber\\
=&:C_1(\tau^{-1}(g,h,x))L^\tau(g,h,x). 
\end{align}

To get a detailed expression of $C_2$ as well as its principal symbol, we define
another invertible transformation $\eta:{\mathbb{R}}^6\to{\mathbb{R}}^6$ by
\[
\eta(y,z,x)=(v,w,x),
\]
where $v=y-z$ and $w=y+z$. Consider the pull-back $C_3:=C_1\circ \eta^{-1}$
satisfying
\begin{eqnarray*}
C_3(v,w,x)&=& C_1(\eta^{-1}(v,w,x))=C_1\Big(\frac{v+w}2,\frac{w-v}2,x\Big)\\
&=& \frac1{(2\pi)^2}\int_{{\mathbb{R}}^2}e^{{\rm i}
v\cdot\xi}c_1\Big(\frac{v+w}2,x,\xi\Big)d\xi
=\frac1{(2\pi)^2}\int_{{\mathbb{R}}^2}e^{{\rm i}
v\cdot\xi}c_3\left(w,x,\xi\right)d\xi,
\end{eqnarray*}
where we have used the properties of symbols (cf. \cite[Lemma 18.2.1]{H07}) and
that $c_3$ has the following asymptotic expansion: 
\begin{eqnarray*}
c_3(w,x,\xi)&= &e^{-{\rm i}\left\langle
D_v,D_{\xi}\right\rangle}c_1\Big(\frac{v+w}2,x,\xi\Big)\Big|_{v=0
}\\
&\sim&\sum_{j=0}^{\infty}\frac{\langle-{\rm i}
D_v,D_\xi\rangle^j}{j!}c_1\Big(\frac{v+w}2,x,\xi\Big)\Big|_{v=0}.
\end{eqnarray*}
Moreover, the principal symbol of $c_3$ is 
\[
c_3^p(w,x,\xi)=c_1^p\Big(\frac{w}2,x,\xi\Big)=\mu\left(\frac{w}
2\right)\theta(x)|\xi|^{-2s}.
\]

Finally, we define a diffeomorphism
$\gamma:=\eta\circ\tau^{-1}:(g,h,x)\mapsto(v,w,x)$, which preserves the plane
$\{(g,h,x)\in{\mathbb{R}}^6 : g=0\}$, i.e., if $g=0$ then $v=0$.
By Theorem 18.2.9 in \cite{H07}, the pull-back $C_4:=C_3\circ\gamma$ can
be calculated by
\[
C_4(g,h,x)=C_3(\gamma(g,h,x))=\frac1{(2\pi)^2}\int_{{\mathbb{R}}^2}e^{{\rm i}
g\cdot\xi}c_4\left(h,x,\xi\right)d\xi,
\]
where
\begin{eqnarray*}
c_4(h,x,\xi)&=&c_3\big(\gamma_2(0,h,x),(\gamma_{11}'(0,h,x))^{-\top}
\xi\big)\left|\det\left(\gamma_{11}'(0,h,x)\right)\right|^{-1}+r_3(h,x,\xi)\\
&=& c_3^p\big(\gamma_2(0,h,x),(\gamma_{11}'(0,h,x))^{-\top}
\xi\big)\left|\det\left(\gamma_{11}'(0,h,x)\right)\right|^{-1}+r_4(h,x,\xi).
\end{eqnarray*}
Here the residuals $r_3,r_4\in S^{-2s-1}$, $\gamma=(\gamma_1,\gamma_2)$ with
$\gamma_1(g,h,x)=v$ and $\gamma_2(g,h,x)=(w,x)$, and $\gamma_{11}'$ is
determined by the Jacobian matrix
\[
\gamma'=\left[\begin{array}{cc}\gamma_{11}'&\gamma_{12}'\\\gamma_{21}'&\gamma_{
22}'\end{array}\right].
\]
Hence, $c_4\in S^{-2s}$ is still $C^{\infty}$-smooth and compactly supported in
the variables $(h,x)$ with the principal symbol
\begin{align}\label{eq:c4}
c_4^p(h,x,\xi)=\mu\Big(\frac{w(0,h,x)}2\Big)\theta(x)\left|(\gamma_{11}
'(0,h,x))^{-\top}\xi\right|^{-2s}\left|\det(\gamma_{11}'(0,h,
x))\right|^{-1}.
\end{align}

Noting that
$C_4=C_3\circ\gamma=C_1\circ\eta^{-1}\circ\eta\circ\tau^{-1}=C_1\circ\tau^{-1}
$ and combining with \eqref{eq:C2}, we obtain 
\begin{eqnarray}\label{eq:C2c5}
&&C_2(g,h,x)=C_4(g,h,x)L^\tau(g,h,x)\nonumber\\
&&=\frac1{(2\pi)^2}\int_{{\mathbb{R}}^2}e^{{\rm i}
g\cdot\xi}c_4\left(h,x,\xi\right)L^\tau(g,h,x)d\xi
=\frac1{(2\pi)^2}\int_{{\mathbb{R}}^2}e^{{\rm i}
g\cdot\xi}c_5(h,x,\xi)d\xi,
\end{eqnarray}
where we have used Lemma 18.2.1 in \cite{H07} again and the fact
that the function $L^\tau(g,h,x)$ is smooth in the domain
$\tau({\mathcal{D}}^\theta)$. Similar to the asymptotic expansion of $c_3$, we
have
\[
c_5(h,x,\xi)\sim\sum_{j=0}^{\infty}\frac{\langle-{\rm i}
D_g,D_\xi\rangle^j}{j!}\left(c_4\left(h,x,\xi\right)L^\tau(g,h,x)\right)\Big|_{
g=0}. 
\]
Using \eqref{eq:c4} and the expression of $L^\tau$ defined in
\eqref{eq:C2}, and residual $r_5:=c_5-c_5^p\in S^{-2s-1}$, we obtain
the principal symbol
\begin{eqnarray}\label{eq:c5}
c_5^p(h,x,\xi)&=&c_4^p\left(h,x,\xi\right)L^\tau(0,h,x)\nonumber\\
&=&\mu\Big(\frac{w(0,h,x)}2\Big)\theta(x)\left|(\gamma_{11}'(0,h,
x))^{-\top}\xi\right|^{-2s}\frac{\det((\tau^{-1})'(0,h,
x))}{\left|\det(\gamma_{11}'(0,h,x))\right|(h\cdot
e_1)^{1+l_1+l_2}}. 
\end{eqnarray}

Let $\alpha=\frac{h_2}{h_1}$. Simple calculations show that 
\begin{eqnarray*}
\gamma_{11}'(0,h,x)&=&\frac{\partial v}{\partial g}(0,h,x)=\left[
\begin{array}{cc}
\frac{\partial v_1}{\partial g_1}&\frac{\partial v_1}{\partial g_2}\\[2pt]
\frac{\partial v_2}{\partial g_1}&\frac{\partial v_2}{\partial g_2}
\end{array}
\right](0,h,x)\\
&=& 2\left[
\begin{array}{cc}
\sin\alpha-\alpha\cos\alpha&\cos\alpha\\
\cos\alpha+\alpha\sin\alpha&-\sin\alpha
\end{array}
\right]
\end{eqnarray*}
is invertible since $\det(\gamma_{11}'(0,h,x))=-4$ and
$\gamma_2(0,h,x)=(w(0,h,x),x)$ with
\[
w(0,h,x)=\Big(2h_1\sin\Big(\frac{h_2}{h_1}\Big)+2x_1,2h_1\cos\Big(\frac{h_2
}{h_1}\Big)+2x_2\Big).
\] 
Moreover, a straightforward calculation gives 
\begin{eqnarray*}
&&\left(\tau^{-1}\right)'(0,h,x)=\frac{\partial(y,z,x)}{\partial(g,h,x)}\Bigg|_{
g=0}\\
&&=\left[\begin{array}{cccccc}
\sin\alpha-\alpha\cos\alpha&\cos\alpha&\sin\alpha-\alpha\cos\alpha&\cos\alpha&1&
0\\
\cos\alpha+\alpha\sin\alpha&-\sin\alpha&\cos\alpha+\alpha\sin\alpha&-\sin\alpha&
0&1\\
-\sin\alpha+\alpha\cos\alpha&-\cos\alpha&\sin\alpha-\alpha\cos\alpha&\cos\alpha&
1&0\\
-\cos\alpha-\alpha\sin\alpha&\sin\alpha&\cos\alpha+\alpha\sin\alpha&-\sin\alpha&
0&1\\
0&0&0&0&1&0\\
0&0&0&0&0&1
\end{array}\right]
\end{eqnarray*}
and $\det((\tau^{-1})'(0,h,x))=4$.

Combining \eqref{eq:u0square2} and \eqref{eq:C2c5}--\eqref{eq:c5}, we obtain 
\begin{eqnarray*}
I_{l_1,l_2}(x;k)&=&\int_{{\mathbb{R}}^2}\int_{{\mathbb{R}}^2}e^{{\rm i}2
\kappa_{\rm r}(e_1\cdot g)-2\kappa_{\rm i}(e_1\cdot h)}\\
&&\qquad \times\left[\frac1{(2\pi)^2}\int_{{\mathbb{R}}^2}e^{{\rm i}
g\cdot\xi}\Big(c_4^p(h,x,\xi)L^\tau(0,h,x)+r_5(h,x,\xi)\Big)
d\xi\right]dgdh\\
&=&\int_{{\mathbb{R}}^2}\int_{{\mathbb{R}}^2}e^{-2\kappa_{\rm i}(e_1\cdot
h)}\Big[c_4^p(h,x,\xi)L^\tau(0,h,x)+r_5(h,x,\xi)\Big]\delta(2
\kappa_{\rm r}e_1+\xi)d\xi dh\\
&=&\int_{{\mathbb{R}}^2}e^{-2\kappa_{\rm i}(e_1\cdot
h)}\Big[c_4^p\left(h,x,-2\kappa_{\rm r}e_1\right)L^\tau(0,h,x)+r_5(h,x,-2
\kappa_{\rm r} e_1)\Big]dh\\
&=&\bigg[\int_{{\mathbb{R}}^2}e^{-2\kappa_{\rm i}(e_1\cdot
h)}\mu\Big(\frac{w(0,h,x)}2\Big)\theta(x)\left|(\gamma_{11}'(0,h,
x))^{-\top}(-2\kappa_{\rm r}e_1)\right|^{-2s}\\
&&\qquad \times\frac{1}{(e_1\cdot
h)^{1+l_1+l_2}}dh+O(\kappa_{\rm r}^{-2s-1})\bigg]\\
&=&\bigg[\int_{{\mathbb{R}}^2}\frac{e^{-2\kappa_{\rm i}(e_1\cdot h)}}{(e_1\cdot
h)^{1+l_1+l_2}}\mu\Big(\frac{w(0,h,x)}2\Big)\theta(x)dh\bigg]\kappa_{\rm
r}^{-2s}+O(\kappa_{\rm r}^{-2s-1})\\
&=&:M_{l_1,l_2}^\kappa(x)\kappa_{\rm r}^{-2s}+O(\kappa_{\rm
r}^{-2s-1}),
\end{eqnarray*}
where we have used the fact that $
\delta(\xi)=\frac1{(2\pi)^d}\int_{\mathbb{R}^d}e^{{\rm i} x\cdot\xi}d\xi
$ in the second step and 
\begin{align*}
M_{l_1,l_2}^\kappa(x)=\int_{{\mathbb{R}}^2}\frac{e^{-2\kappa_{\rm i}(e_1\cdot
h)}}{(e_1\cdot h)^{1+l_1+l_2}}\mu\Big(\frac{w(0,h,x)}2\Big)\theta(x)dh. 
\end{align*}
To simplify the expression of $M^\kappa_{l_1,l_2}(x)$, we consider another
coordinate transformation $\rho:{\mathbb{R}}^2\to{\mathbb{R}}^2$ defined by 
\[
\rho(h)=\zeta:=\left(h_1\sin\left(\frac{h_2}{h_1}\right),h_1\cos\left(\frac{h_2}
{h_1}\right)\right)+x,
\]
which has the Jacobian
\[
\text{det}(\rho')=\left|\begin{array}{cc}
\sin\left(\frac{h_2}{h_1}\right)-\frac{h_2}{h_1}\cos\left(\frac{h_2}{h_1}
\right)&\cos\left(\frac{h_2}{h_1}\right)\\[2pt]
\cos\left(\frac{h_2}{h_1}\right)+\frac{h_2}{h_1}\sin\left(\frac{h_2}{h_1}
\right)&-\sin\left(\frac{h_2}{h_1}\right)
\end{array}\right|=-1.
\]
Noting that $\det((\rho^{-1})')=\frac1{\det(\rho')}=-1$, we get
\[
M_{l_1,l_2}^\kappa(x)=\int_{{\mathbb{R}}^2}\frac{e^{-2\kappa_{\rm i}|x-\zeta|}}{
|x-\zeta|^{1+l_1+l_2}}\mu\left(\zeta\right)d\zeta,\quad x\in{\mathcal{U}}.
\]
Combining the above estimates, we obtain 
\begin{eqnarray*}
{\mathbb{E}}|u^2(x;k)|^2&=&\frac{|a_{0}|^2}{16|\kappa|}I_{0,0}(x;k)
+\Re\left[\frac{a_{0}\bar{a}_1}{8|\kappa|\kappa}I_{0,1}(x;k)\right]+\frac
{|a_{1}|^2}{16|\kappa|^3}I_{1,1}(x;k)\\
&&\quad +\Re\left[\frac{a_{0}\bar{a}_2}{8|\kappa|\kappa^2}I_{0,2}
(x;k)\right ]
+\Re\left[\frac{a_{1}\bar{a}_2}{8|\kappa|^3\kappa}I_{1,2}(x;k)\right]
+\frac{|a_{2}|^2}{16|\kappa|^5}I_{2,2}(x;k)\\
&=&\frac{|a_{0}|^2}{16|\kappa|}\left[M_{0,0}^\kappa(x)\kappa_{\rm r}^{-2s}
+O(\kappa_{\rm r}^{-2s-1})\right]\\
&&\quad +\Re\left[\frac{a_{0}\bar{a}_1}{8|\kappa|\kappa}\left(M_{0,1}
^\kappa(x)\kappa_{\rm
r}^{-2s}+O(\kappa_{\rm r}^{-2s-1})\right)\right]\\
&&\quad +\frac{|a_{1}|^2}{16|\kappa|^3}\left[M_{1,1}^\kappa(x)\kappa_{\rm
r}^{-2s}+O(\kappa_{\rm r}^{-2s-1})\right]\\
&&\quad +\Re\left[\frac{a_{0}\bar{a}_2}{8|\kappa|\kappa^2}\left(M_{0,2}
^\kappa(x)\kappa_{\rm
r}^{-2s}+O(\kappa_{\rm r}^{-2s-1})\right)\right]\\
&&\quad +\Re\left[\frac{a_{1}\bar{a}_2}{8|\kappa|^3\kappa}\left(M_{1,2}
^\kappa(x)\kappa_{\rm r}^{-2s}+O(\kappa_{\rm r}^{-2s-1})\right)\right]\\
&&\quad +\frac{|a_{2}|^2}{16|\kappa|^5}\left[M_{2,2}^\kappa(x)\kappa_{\rm
r}^{-2s} +O(\kappa_{\rm r}^{-2s-1})\right]\\
&=&\frac{|a_{0}|^2}{16}M_{0,0}^\kappa(x)|\kappa|^{-1}\kappa_{\rm r}^{-2s}
+O(\kappa_{\rm r}^{-2s-2}),
\end{eqnarray*}
which completes the proof.
\end{proof}

\begin{theorem}\label{tm:main2d}
Let $f\in L^2(\Omega,W^{H-\epsilon,p})$ with $H,\epsilon$, and $p$
satisfying the conditions given in Theorem \ref{tm:solution}. Then for any
$x\in{\mathcal{U}}$,
\[
\lim_{k\to\infty}k^{2s+1}{\mathbb{E}}|u(x;k)|^2=\frac{1}{2^3\pi}\int_{{\mathbb{R
}}^2}\frac{e^{-\sigma|x-y|}}{|x-y|}\mu\left(y\right)dy=:T(x).
\]
\end{theorem}

\begin{proof}
Note that
\begin{eqnarray*}
k^{2s+1}{\mathbb{E}}|u(x;k)|^2&=&k^{2s+1}{\mathbb{E}}|u^2(x;k)|^2+2k^{2s+1}{
\mathbb{E}}\Re\Big[\overline{u^2(x;k)}(u(x;k)-u^2(x;k))\Big]\\
&&\quad+k^{2s+1}{\mathbb{E}}\left|u(x;k)-u^2(x;k)\right|^2\\
&=&:V_1(k)+V_2(k)+V_3(k).
\end{eqnarray*}
Next we calculate the limits of $V_1,V_2$, and $V_3$, respectively.

Using the asymptotic expansions of the Hankel function in \eqref{eq:Hn}, we get
\[
\big|H_n^{(1)}(\kappa|x-y|)-H_{n,N}^{(1)}
(\kappa|x-y|)\big|=O\big(|\kappa|x-y||^{-(N+\frac32)}\big),\quad
k\to\infty.
\]
Noting $H_0^{(1)'}(z)=-H_1^{(1)}(z)$, we have
\begin{align*}
\left|\partial_{y_i}H_0^{(1)}(\kappa|x-y|)-\partial_{y_i}H_{0,N}^{(1)}
(\kappa|x-y|)\right|=O\big(|\kappa|^{-(N+\frac12)}|x-y|^{
-(N+\frac32)}\big),\quad k\to\infty.
\end{align*}
Hence
\begin{eqnarray*}
&&{\mathbb{E}}|u(x;k)-u^2(x;k)|^2={\mathbb{E}}\left|\int_{\mathcal{D}}
\left(\Phi_\kappa(x,y)-\Phi_\kappa^2(x,y)\right)f(y)dy\right|^2\\
&&\lesssim\|\Phi_\kappa(x,\cdot)-\Phi_\kappa^2(x,\cdot)\|_{W^{1,q}({\mathcal{D}}
)}^2{\mathbb{E}}\|f\|_{W^{-1,p}({\mathcal{D}})}^2\\
&&\lesssim\|\Phi_\kappa(x,\cdot)-\Phi_\kappa^2(x,\cdot)\|_{W^{1,q}({\mathcal{D}}
)}^2{\mathbb{E}}\|f\|_{W^{H-\epsilon,p}({\mathcal{D}})}^2
\lesssim|\kappa|^{-5},
\end{eqnarray*}
where $f\in L^2(\Omega,W_{\rm comp}^{H-\epsilon,p})\subset
L^2(\Omega,W^{-1,p}_{\rm comp})$ for $H\in(\frac dp-2,0]$ and
$p\in\left(1,2\right]$ and $\frac1q+\frac1p=1$ according to Theorem
\ref{tm:solution} with $d=2$.
It then indicates that
\[
V_3(k)\lesssim k^{2s+1}|\kappa|^{-5}=k^{2s+1}(k^4+k^2\sigma^2)^{-\frac54}\to0
\]
as $k\to\infty$ since $s<2$ for $d=2$.

For $V_2(k)$, we have
\begin{align*}
V_2(k)\le2\left(k^{2s+1}{\mathbb{E}}|u^2(x;k)|^2\right)^{\frac12}\left(k^{2s+1}{
\mathbb{E}}|u(x;k)-u^2(x;k)|^2\right)^{\frac12}=2V_1(k)^{\frac12}V_3(k)^{\frac12
},
\end{align*}
which converges to $0$ if the limit of $V_1(k)$ exists.

For $V_1(k)$, by Proposition \ref{prop:u1},
\begin{align*}
V_1(k)=T_\kappa(x)k^{2s+1}|\kappa|^{-1}\kappa_{\rm
r}^{-2s}+O(k^{2s+1}\kappa_{\rm r}^{-2s-2}).
\end{align*}
We have from \eqref{kri} that
\[
\lim_{k\to\infty}V_1(k)=\lim_{k\to\infty}T_\kappa(x)=\frac{|a_0|^2}{16}\int_
{{\mathbb{R}}^2}\frac{e^{-\sigma|x-y|}}{|x-y|}\mu\left(y\right)dy,
\]
which completes the proof.
\end{proof}

\begin{remark}
It can be seen from the above proof that only two terms are needed in the
truncation of \eqref{eq:Hn} if the source is extremely rough with $s\in[0,\frac
d2)$. More
precisely, it suffices to consider the approximate solution
\[
u^1(x;k):=-\int_{\mathbb{R}^d}\Phi_\kappa^1(x,y)f(y)dy
\]
instead of $u^2$, where $V_3(k)\lesssim k^{2s+1}|\kappa|^{-3}\to0$ as
$k\to\infty$ since $s<\frac d2=1$.
\end{remark}

\begin{theorem}\label{tm:mu2d}
The strength $\mu$ is uniquely determined by
\[
T(x)=\frac{1}{2^3\pi}\int_{{\mathbb{R}}^2}\frac{e^{-\sigma|x-y|}}{|x-y|}
\mu\left(y\right)dy,\quad x\in{\mathcal{U}}.
\]
\end{theorem}

\begin{proof}
We first consider the function $ V(x):=e^{-\sigma|x|}/|x|^l$
for some positive number $\sigma$ and integer $l\ge1$, which can be regarded as
a composition of functions $U(s)=e^{-\sigma s}/s^l$ and $r(x)=|x|$,
i.e., $V(x)=U(r(x))$. A simple calculation shows that
\begin{eqnarray*}
\Delta V(x)&=&U''(r(x))\nabla r(x)\cdot\nabla r(x)+U'(r(x))\Delta r(x)\\
&=&\left[\frac{\sigma^2}{|x|^l}+\frac{2l\sigma}{|x|^{l+1}}+\frac{l(l+1)}{|x|^{
l+2}}\right]e^{-\sigma|x|}+\left[\frac{-\sigma}{|x|^l}+\frac{-l}{|x|^{l+1}}
\right]e^{-\sigma|x|}
\frac1{|x|}\\
&=&\left[\frac{l^2}{|x|^{l+2}}+\frac{(2l-1)\sigma}{|x|^{l+1}}+\frac{\sigma^2}{
|x|^l}\right]e^{-\sigma|x|}.
\end{eqnarray*}
Hence, if $T(x)$ is known in ${\mathcal{U}}$, then so is $\Delta^n T(x)$ for any
$n\in\mathbb{N}$. 
It implies that the following integral is determined by the measurement $T(x)$:
\begin{eqnarray*}
&&\int_{{\mathcal{D}}}P\Big(\frac1{|x-y|}\Big)\frac{e^{-\sigma|x-y|}}{|x-y|}
\mu(y)dy\\
&=&\int_{r_1}^{r_2}P\Big(\frac1r\Big)\frac{e^{-\sigma
r}}{r}\bigg[\int_{|x-y|=r}\mu(y)ds(y)\bigg]dr\\
&=&\int_{r_1^{-1}}^{r_2^{-1}}P(t)\frac{e^{-\sigma
t^{-1}}}{t^{-1}}\bigg[\int_{|x-y|=t^{-1}}\mu(y)ds(y)\bigg]\left(-\frac1{t^2}
\right)dt\\
&=&\int_{r_2^{-1}}^{r_1^{-1}}P(t)\frac{e^{-\sigma
t^{-1}}}{t}\bigg[\int_{|x-y|=t^{-1}}\mu(y)ds(y)\bigg]dt,
\end{eqnarray*}
where $P(t)=\sum_{j=0}^Jc_jt^j$ is any polynomial of order $J\in\mathbb{N}$ with
real numbers $c_j$, $j=0,\cdots,J$, $r_1=\min_{y\in{\mathcal{D}}}|x-y|\ge r_0>0$
and $r_2=\max_{y\in{\mathcal{D}}}|x-y|$.

Denote $S(x,r)=\int_{|x-y|=r}\mu(y)ds(y)$, which is continuous and compactly
supported on $[r_1,r_2]$.
Since the polynomial space on the interval $[r_2^{-1},r_1^{-1}]$ is dense in
$C([r_2^{-1},r_1^{-1}])$, the function $\frac{e^{-\sigma t^{-1}}}{t}S(x,t^{-1})$
can be uniquely determined on $[r_2^{-1},r_1^{-1}]$, and so does $S(x,t^{-1})$.
Hence $S(x,r)$ can be uniquely determined on $[r_1,r_2]$.

To recover the strength $\mu$ based on $S(x,t)$, the classical deconvolution
is used. More precisely, we consider the convolution between
$\mu$ and $g(x)=e^{-\frac{|x|^2}2}$:
\[
(g*\mu)(x)=\int_{r_1}^{r_2}e^{-\frac{r^2}2}S(x,r)dr,
\]
which is known since $S(x,r)$ can be recovered. Then the Fourier transform
yields
\[
\mathcal{F}[\mu](\xi)=\frac{\mathcal{F}[g*\mu](\xi)}{\mathcal{F}[g](\xi)}=e^{
-\frac{|\xi|^2}2}\mathcal{F}[g*\mu](\xi),
\]
which implies that $\mu$ can be uniquely determined. 
\end{proof}

\subsection{Three-dimensional case}

Now we consider $d=3$. By Theorem \ref{tm:solution}, the solution of
the direct problem is
\begin{equation}\label{dp3d}
u(x;k)=-\frac1{4\pi}\int_{{\mathbb{R}}^3}\frac{e^{{\rm i}\kappa|x-y|}}{
|x-y|}f(y)dy.
\end{equation}
Following the same procedure as that for the two-dimensional case, we first show
that the strength $\mu$ is uniquely determined by the variance of the solution
$u$.

\begin{theorem}\label{tm:main3d}
Assume that $f\in L^2(\Omega,W^{H-\epsilon,p})$ with $H,\epsilon$ and $p$
satisfying the conditions given in Theorem \ref{tm:solution}. Then for any
$x\in\mathcal U$, 
\[
\lim_{k\to\infty}k^{2s}{\mathbb{E}}|u(x;k)|^2=\frac{1}{2^4\pi^2}\int_{{\mathbb{R
}}^3}\frac{e^{-\sigma|x-y|}}{|x-y|^2}\mu\left(y\right)dy=:\tilde T(x).
\]
\end{theorem}

\begin{proof}
Using \eqref{dp3d}, we have for any $x\in{\mathcal{U}}$ that 
\begin{align*}
{\mathbb{E}}|u(x;k)|^2=&\frac1{16\pi^2}\int_{{\mathbb{R}}^3}\int_{{\mathbb{R}}^3
}\frac{e^{{\rm i}\kappa|x-y|-{\rm i}\bar\kappa|x-z|}}{|x-y||x-z|}
{\mathbb{E}}[f(y)f(z)]dydz\\
=&\frac1{16\pi^2}\int_{{\mathbb{R}}^3}\int_{{\mathbb{R}}^3}\frac{e^{{\rm i}
\kappa|x-y|-{\rm i}\bar\kappa|x-z|}}{|x-y||x-z|}K_f(y,
z)\theta(x)dydz\\
=&\frac1{16\pi^2}\int_{{\mathbb{R}}^3}\int_{{\mathbb{R}}^3}e^{{\rm i}
\kappa_{\rm
r}\left(|x-y|-|x-z|\right)-\kappa_{\rm i}\left(|x-y|+|x-z|\right)}\frac{ C_1(y ,
z,x)}{|x-y||x-z|}dydz,
\end{align*}
where $\theta_0^{\infty}$ such that $\theta|_{{\mathcal{U}}}\equiv1$ and
supp$(\theta)\subset{\mathbb{R}}^3\backslash\overline{{\mathcal{D}}}$, 
\[
C_1(y,z,x):=K_f(y,z)\theta(x)=\frac1{(2\pi)^3}\int_{{\mathbb{R}}^3}e^{{\rm
i}(y-z)\cdot\xi}c_1(y,x,\xi)d\xi.
\]
Here $c_1(y,x,\xi):=c(y,\xi)\theta(x)$ with the symbol $c(y,\xi)$ satisfying
\eqref{eq:Kf}. Then the principal symbol of $c_1$ has the form
\[
c_1^p(y,x,\xi)=\mu(y)\theta(x)|\xi|^{-2s}.
\]

We first define an invertible
transformation $\tau:{\mathbb{R}}^9\to{\mathbb{R}}^9$ by $
\tau(y,z,x)=(g,h,x)$,
where $g=(g_1,g_2,g_3)$ and $h=(h_1,h_2,h_3)$ with
\begin{eqnarray*}
&& g_1=\frac12\left(|x-y|-|x-z|\right),\quad
h_1=\frac12\left(|x-y|+|x-z|\right),\\
&&g_2=\frac12\bigg[|x-y|\arccos\Big(\frac{y_3-x_3}{|x-y|}
\Big)-|x-z|\arccos\Big(\frac{z_3-x_3}{|x-z|}\Big)\bigg],\\
&&h_2=\frac12\bigg[|x-y|\arccos\Big(\frac{y_3-x_3}{|x-y|}
\Big)+|x-z|\arccos\Big(\frac{z_3-x_3}{|x-z|}\Big)\bigg],\\
&&g_3=\frac12\bigg[|x-y|\arctan\Big(\frac{y_2-x_2}{y_1-x_1}
\Big)-|x-z|\arctan\Big(\frac{z_2-x_2}{z_1-x_1}\Big)\bigg],\\
&&h_3=\frac12\bigg[|x-y|\arctan\Big(\frac{y_2-x_2}{y_1-x_1}
\Big)+|x-z|\arctan\Big(\frac{z_2-x_2}{z_1-x_1}\Big)\bigg].
\end{eqnarray*}
Then 
\[
{\mathbb{E}}|u(x;k)|^2=\frac1{16\pi^2}\int_{{\mathbb{R}}^3}\int_{{\mathbb{R}}^3}
e^{2{\rm i}\kappa_{\rm r}(e_1\cdot g)-2\kappa_{\rm i}(e_1\cdot
h)}C_2(g,h,x)dgdh,
\]
where $e_1=(1,0,0)$ and
\begin{eqnarray*}
C_2(g,h,x)&=&C_1(\tau^{-1}(g,h,x))\frac{\text{det}\left((\tau^{-1})'(g,h,
x)\right)}{((g+h)\cdot e_1)((h-g)\cdot e_1)}\\
&=&:C_1(\tau^{-1}(g,h,x))L^\tau(g,h,x).
\end{eqnarray*}

Next is to get an explicit expression of $C_2$ with respect to $(g,h,x)$.
We define another invertible transformation
$\eta:{\mathbb{R}}^9\to{\mathbb{R}}^9$
by $\eta(y,z,x)=(v,w,x)$ with $v=y-z$ and $w=y+z$, and define the diffeomorphism
$\gamma:=\eta\circ\tau^{-1}:(g,h,x)\mapsto(v,w,x)$.
Following the same procedure as that used in Proposition \ref{prop:u1}, 
by defining $C_3:=C_1\circ\eta^{-1}$, we obtain
\begin{eqnarray*}
C_3(v,w,x)&=&C_1(\eta^{-1}(v,w,x))=C_1\Big(\frac{v+w}2,\frac{w-v}2,x\Big)\\
&=&\frac1{(2\pi)^3}\int_{\mathbb{R}^3}e^{{\rm
i}v\cdot\xi}c_1\Big(\frac{v+w}2,x,\xi\Big)d\xi
=\frac1{(2\pi)^3}\int_{\mathbb{R}^3}e^{{\rm
i}v\cdot\xi}c_3\left(w,x,\xi\right)d\xi,
\end{eqnarray*}
where $c_3$ has the principal symbol
$c_3^p(w,x,\xi)=c_1^p\left(\frac{v+w}2,x,\xi\right)|_{v=0}=\mu(\frac
w2)|\xi|^{-2s}\theta(x)$. By Theorem 18.2.9 in \cite{H07}, 
\begin{align*}
C_4(g,h,x):=C_3\circ\gamma(g,h,x)=\frac1{(2\pi)^3}\int_{\mathbb{R}^3}e^{{\rm 
i}g\cdot\xi}c_4(h,x,\xi)d\xi,
\end{align*}
where $c_4$ has the principal symbol
\[
c_4^p(h,x,\xi)=c_3^p\left(\gamma_2(0,h,x),\left(\gamma_{11}'(0,h,x)\right)^{
-\top}\xi\right)\left|\det\left(\gamma_{11}'(0,h,x)\right)\right|^{-1},
\]
and $\gamma_2(0,h,x)=(w(0,h,x),x), \gamma_{11}'(0,h,x)=\frac{\partial
v}{\partial g}(0,h,x)$. Noting that
$C_4=C_3\circ\gamma=C_1\circ\eta^{-1}\circ\eta\circ\tau^{-1}=C_1\circ\tau^{-1}$,
we are able to give the expression of $C_2$:
\begin{eqnarray*}
C_2(g,h,x)&=&C_1\circ\tau^{-1}(g,h,x)L^{\tau}(g,h,x)\\
&=&\frac1{(2\pi)^3}\int_{{\mathbb{R}}^3}e^{{\rm i} g\cdot
\xi}c_4(h,x,\xi)L^\tau(g,h,x)d\xi
=\frac1{(2\pi)^3}\int_{{\mathbb{R}}^3}e^{{\rm i} g\cdot
\xi}c_5(h,x,\xi)d\xi,
\end{eqnarray*}
where the principal symbol of $c_5$ is 
\begin{eqnarray*}
&&c_5^p(h,x,\xi)=c_4^p(h,x,\xi)L^{\tau}(0,h,x)
=\mu\Big(\frac{w(0,h,x)}2\Big)\theta(x)\\
&&\qquad \times\Big|\Big(\frac{\partial
v}{\partial
g}(0,h,x)\Big)^{-\top}\xi\Big|^{-2s}\Big|{\rm det}\Big(\frac{\partial
v}{\partial
g}(0,h,x)\Big)\Big|^{-1}\frac{\text{det}\left((\tau^{-1})'(0,h,x)\right)}{
(h\cdot e_1)^2}
\end{eqnarray*}
and the residual $r_5:=c_5-c_5^p\in S^{-2s-1}$. 

It then suffices to calculate $c_5^p$. Noting that
\begin{eqnarray*}
&&h_1+g_1=|x-y|,\quad h_1-g_1=|x-z|,\\
&&\frac{h_2+g_2}{h_1+g_1}=\arccos\Big(\frac{y_3-x_3}{|x-y|}\Big),\quad\frac{
h_2-g_2}{h_1-g_1}=\arccos\Big(\frac{z_3-x_3}{|x-z|}\Big),\\
&&\frac{h_3+g_3}{h_1+g_1}=\arctan\Big(\frac{y_2-x_2}{y_1-x_1}\Big),\quad\frac
{h_3-g_3}{h_1-g_1}=\arctan\Big(\frac{z_2-x_2}{z_1-x_1}\Big),
\end{eqnarray*}
we get
\begin{eqnarray*}
y_1&=&x_1+(h_1+g_1)\sin\left(\frac{h_2+g_2}{h_1+g_1}\right)\cos\left(\frac{
h_3+g_3}{h_1+g_1}\right),\\
y_2&=&x_2+(h_1+g_1)\sin\left(\frac{h_2+g_2}{h_1+g_1}\right)\sin\left(\frac{
h_3+g_3}{h_1+g_1}\right),\\
y_3&=&x_3+(h_1+g_1)\cos\left(\frac{h_2+g_2}{h_1+g_1}\right),\\
z_1&=&x_1+(h_1-g_1)\sin\left(\frac{h_2-g_2}{h_1-g_1}\right)\cos\left(\frac{
h_3-g_3}{h_1-g_1}\right),\\
z_2&=&x_2+(h_1-g_1)\sin\left(\frac{h_2-g_2}{h_1-g_1}\right)\sin\left(\frac{
h_3-g_3}{h_1-g_1}\right),\\
z_3&=&x_3+(h_1-g_1)\cos\left(\frac{h_2-g_2}{h_1-g_1}\right).\\
\end{eqnarray*}
A simple calculation yields that 
\[
\frac{\partial v}{\partial g}(0,h,x)=
2\left[
\begin{array}{ccc}
\sin\alpha\cos\beta-\alpha\cos\alpha\cos\beta+\beta\sin\alpha\sin\beta&\cos
\alpha\cos\beta&-\sin\alpha\sin\beta\\
\sin\alpha\sin\beta-\alpha\cos\alpha\sin\beta-\beta\sin\alpha\cos\beta&\cos
\alpha\sin\beta&\sin\alpha\cos\beta\\
\cos\alpha+\alpha\sin\alpha&-\sin\alpha&0
\end{array}
\right],
\]
where $\alpha:=\frac{h_2}{h_1}, \beta:=\frac{h_3}{h_1}$, and
\[
(\tau^{-1})'(0,h,x)=\left[
\begin{array}{ccc}
\frac12\frac{\partial v}{\partial g}&\frac12\frac{\partial v}{\partial
g}&I\\[4pt]
-\frac12\frac{\partial v}{\partial g}&\frac12\frac{\partial v}{\partial g}&I\\
0&0&I
\end{array}
\right].
\]
Here $I$ is the $3\times 3$ identity matrix. It can be verified that 
\[
\text{det}\Big(\frac{\partial v}{\partial g}(0,h,x)\Big)=8\sin\alpha,\quad 
L^\tau(0,h,x)=\frac{8\sin^2\alpha}{(h\cdot e_1)^2}, 
\]
and
\begin{align*}
\Big(\frac{\partial v}{\partial g}(0,h,x)\Big)^{-\top}=\frac12\left[
\begin{array}{ccc}
\sin\alpha\cos\beta&\cos\alpha\cos\beta+\alpha\sin\alpha\cos\beta&-\frac{
\sin\beta}{\sin\alpha}+\beta\sin\alpha\cos\beta\\[2pt]
\sin\alpha\sin\beta&\cos\alpha\sin\beta+\alpha\sin\alpha\sin\beta&\frac{
\cos\beta}{\sin\alpha}+\beta\sin\alpha\sin\beta\\
\cos\alpha&-\frac{\cos\beta}{\sin\alpha}
-\beta\sin\alpha\sin\beta&\beta\cos\alpha
\end{array}
\right].
\end{align*}
We then have 
\begin{eqnarray*}
{\mathbb{E}}|u(x;k)|^2&=&\frac1{16\pi^2}\int_{{\mathbb{R}}^3}\int_{{\mathbb{R}}
^3}e^{2{\rm i}\kappa_{\rm r}(e_1\cdot g)-2\kappa_{\rm i}(e_1\cdot
h)}C_2(g,h,x)dgdh\\
&=&\frac1{16\pi^2}\int_{{\mathbb{R}}^3}\int_{{\mathbb{R}}^3}e^{2{\rm i}
\kappa_{\rm r}(e_1\cdot g)-2\kappa_{\rm i}(e_1\cdot
h)}\frac1{(2\pi)^3}\int_{{\mathbb{R}}^3}e^{{\rm i} g\cdot
\xi}c_5(h,x,\xi)d\xi dgdh\\
&=&\frac1{16\pi^2}\int_{{\mathbb{R}}^3}e^{-2\kappa_{\rm i}(e_1\cdot
h)}c_5(h,x,-2\kappa_{\rm r}e_1)dh\\
&=&\frac1{16\pi^2}\int_{{\mathbb{R}}^3}e^{-2\kappa_{\rm i}(e_1\cdot
h)}\bigg[\mu\Big(\frac{w(0,h,x)}2\Big)\theta(x)\kappa_{\rm r}^{-2s}\frac{
\sin\alpha}{(h\cdot e_1)^2}+r_5(h,x,-2\kappa_{\rm r}e_1)\bigg]dh,
\end{eqnarray*}
where 
\[
\frac{w(0,h,x)}2=(h_1\sin\alpha\cos\beta,h_1\sin\alpha\sin\beta,
h_1\cos\alpha)+x.
\]

Define another coordinate transform $\rho:{\mathbb{R}}^3\to{\mathbb{R}}^3$ by
\[
\rho(h)=\zeta:=(h_1\sin\alpha\cos\beta,h_1\sin\alpha\sin\beta,h_1\cos\alpha)+x.
\]
By noting that $|\zeta-x|=h_1=h\cdot e_1$ and
$\det((\rho^{-1})')=\frac1{\det(\rho')}$ with 
\[
\rho'=\left[
\begin{array}{ccc}
\sin\alpha\cos\beta-\alpha\cos\alpha\cos\beta+\beta\sin\alpha\sin\beta&\cos
\alpha\cos\beta&-\sin\alpha\sin\beta\\
\sin\alpha\sin\beta-\alpha\cos\alpha\sin\beta-\beta\sin\alpha\cos\beta&\cos
\alpha\sin\beta&\sin\alpha\cos\beta\\
\cos\alpha+\alpha\sin\alpha&-\sin\alpha&0
\end{array}
\right],
\] 
the data ${\mathbb{E}}|u(x;k)|^2$ turns to be
\begin{align*}
{\mathbb{E}}|u(x;k)|^2=&\left[\frac1{2^4\pi^2}\int_{{\mathbb{R}}^3}\frac{e^{
-2\kappa_{\rm i}|\zeta-x|}}{|\zeta-x|^2}\mu(\zeta)\theta(x)dh\right]\kappa_{\rm
r}^{-2s}+O(\kappa_{\rm r}^{-2s-1}).
\end{align*}

Finally, for any $x\in{\mathcal{U}}$, we have from \eqref{kri} that 
\[
\lim_{k\to\infty}k^{2s}{\mathbb{E}}|u(x;k)|^2=\lim_{k\to\infty}\frac1{2^4\pi^2}
\int_{{\mathbb{R}}^3}\frac{e^{-2\kappa_{\rm i}|\zeta-x|}}{|\zeta-x|^2}
\mu(\zeta)dh\left(\frac{k}{\kappa_{\rm r}}\right)^{2s}=\tilde T(x),
\]
which completes the proof.
\end{proof}

Repeating basically the same proof as that of Theorem \ref{tm:mu2d}, we may
show the uniqueness of the inverse problem in three dimensions. 

\begin{theorem}
The strength $\mu$ is uniquely determined by
\[
\tilde
T(x)=\frac{1}{2^4\pi^2}\int_{{\mathbb{R}}^3}\frac{e^{-\sigma|x-y|}}{|x-y|^2}
\mu\left(y\right)dy,\quad x\in{\mathcal{U}}.
\]
\end{theorem}

\subsection{The case $\sigma=0$ and ergodicity}

If $\sigma=0$, the model \eqref{eq:model} reduces to the one
considered in \cite{LHL}. In this case, the ergodicity of the solution can be
obtained by following the same way which was investigated in \cite{LHL,LPS08}.
This result makes it possible to uniquely recover the strength $\mu$ by a single
realization of the measurements.

\begin{proposition}
Assume that $f\in L^2(\Omega,W^{H-\epsilon,p})$ with $H,\epsilon$ and $p$
satisfying the conditions given in Theorem \ref{tm:solution}.  Let $s=H+\frac
d2$. Then
\begin{itemize}
\item[(i)] if $d=2$,
\begin{align*}
\lim_{K\to\infty}\frac1{K-1}\int_1^Kk^{2s+1}|u(x;k)|^2dk=T(x)\quad\text{a.s.,}
\end{align*}
\item[(ii)] if $d=3$,
\begin{align*}
\lim_{K\to\infty}\frac1{K-1}\int_1^Kk^{2s}|u(x;k)|^2dk=\tilde
T(x)\quad\text{a.s.,}
\end{align*}
\end{itemize}
where $T$ and $\tilde T$ are defined in Theorems \ref{tm:main2d} and
\ref{tm:main3d}, respectively.
\end{proposition} 

\begin{proof}
If $\sigma=0$, following the same procedure as that of Lemma 3.4 in \cite{LHL}
or Proposition \ref{prop:u1}, we may obtain for any $k_1,k_2\ge1$ that 
\begin{eqnarray*}
\left|\mathbb{E}\left[u^2(x;k_1)\overline{u^2(x;k_2)}\right]\right| &\le&
C(1+|k_1-k_2|)^{-2s},\\
\left|\mathbb{E}\left[u^2(x;k_1)u^2(x;k_2)\right]\right|&\le&
C(1+|k_1-k_2|)^{-2s}. 
\end{eqnarray*}
which, together with the fact that
\[
\lim_{K\to\infty}\frac1{K-1}\int_1^KX(t)dt=0,\quad a.s.,
\]
if $|\mathbb{E}X(t_1)X(t_2)|\le C(1+|t_1-t_2|)^{-\varepsilon}$ for
a centered real-valued stochastic process $X$ with continuous paths and some
$\varepsilon>0$ (cf. \cite{CL67,LHL,LPS08}), one can get the desired
results by following the proof in Theorem 3.10 in \cite{LHL}. The details are
omitted for brevity. 
\end{proof}

\section{Conclusion}
\label{sec:conc}

We have studied the inverse random source scattering problem for
the Helmholtz equations with attenuation. The source is assumed to be a
fractional Gaussian random field. The relationship is established between
the fractional Gaussian fields and the generalized Gaussian random fields.
The well-posedness of the direct problem is examined. For the inverse problem,
we show that the micro-correlation strength of the random source can be uniquely
determined by the passive measurement of the wave fields. 

There are some future works which can be considered. For instance, if the
medium is inhomogeneous, the solution cannot be expressed explicitly through
the fundamental solution. The present method is not applicable, a new approach
is needed. Another interesting problem is to consider that both the medium and
the source are random functions. Similar problems for the Schr\"odinger
equation were investigated in \cite{LLM1,LLM2}. The Helmholtz equation is more
difficult because of the coupling of the medium with the wavenumber. It is an
open problem for the Maxwell equations with a random source. The singularity of
Green's tensor may limit the roughness of the source. We hope to be able
to report the progress on these problems elsewhere in the future.


\begin{thebibliography}{10}

\bibitem{AS92}
\textsc{M. Abramowitz and I. Stegun}, {\em Tables of Mathematical
Functions}, Dover, New York, 1970. 

\bibitem{AF03}
\textsc{R. Adams and J. Fournier}, {\em Sobolev Spaces,} 2nd ed., Academic
Press, Amsterdam, 2003.

\bibitem{ABF02}
\textsc{H. Ammari, G. Bao, and J. Fleming}, {\em An inverse source problem
for Maxwell's equations in magnetoencephalography,} SIAM J. Appl. Math., 62
(2002), pp. 1369--1382.

\bibitem{BCL16}
\textsc{G. Bao, C. Chen, and P. Li}, {\em Inverse random source
scattering problems in several dimensions,}
SIAM/ASA J. Uncertain. Quantif., 4 (2016), pp. 1263--1287.

\bibitem{BCL17}
\textsc{G. Bao, C. Chen, and P. Li}, {\em Inverse random source scattering for
elastic waves}, SIAM J. Numer. Anal., 55 (2017), pp. 2616--2643.

\bibitem{BCLZ14}
\textsc{G. Bao, S.-N. Chow, P. Li, and H. Zhou}, {\em An inverse random
source problem for the Helmholtz equation,} Math. Comp., 83 (2014), pp.
215--233.

\bibitem{BLZ19}
\textsc{G. Bao, P. Li, and Y. Zhao}, {\em Stability for the inverse source
problems in elastic and electromagnetic waves}, J. Math. Pures Appl., to
appear.

\bibitem{BLT10}
\textsc{G. Bao, J. Lin, and F. Triki}, {\em A multi-frequency inverse source
problem}, J. Differential Equations, 249 (2010), pp. 3443--3465. 

\bibitem{CHL19}
\textsc{P. Caro, T. Helin, and M. Lassas}, {\em Inverse scattering for a
random potential,} Anal. Appl. (Singap.), 17 (2019), pp. 513--567.

\bibitem{CK13}
\textsc{D. Colton and R. Kress}, {\em Inverse Acoustic and Electromagnetic Scattering Theory,} 3rd ed., Springer, Berlin, 2013.

\bibitem{CL05}
\textsc{Z. Chen and X. Liu}, {\em An adaptive perfectly matched layer
technique for time-harmonic scattering problems,} SIAM J. Numer. Anal., 43
(2005), pp. 645--671.

\bibitem{CL67}
\textsc{H. Cramer and M. Leadbetter}, {\em Stationary and Related
Stochastic Processes,} New York: John Wiley and Sons, 1967.


\bibitem{D79}
\textsc{A. Devaney}, {\em The inverse problem for random sources}, J. Math.
Phys., 20 (1979), 1687--1691. 


\bibitem{FKM04}
\textsc{A. S. Fokas, Y. Kurylev, and V. Marinakis}, {\em The unique
determination of neuronal currents in the brain via magnetoencephalography,}
Inverse Problems, 20 (2004), pp. 1067--1082.

\bibitem{HLO14}
\textsc{T. Helin, M. Lassas, and L. Oksanen}, {\em Inverse problem for the
wave equation with a white noise source,} Comm. Math. Phys., 332 (2014), pp.
933--953.

\bibitem{H03}
\textsc{L. H\"ormander}, {\em The analysis of linear partial differential
operators I,} Classics in Mathematics, Springer-Verlag, Berlin, 2003.

\bibitem{H07}
\textsc{L. H\"ormander}, {\em The analysis of linear partial differential
operators III,} Classics in Mathematics, Springer, Berlin, 2007.

\bibitem{IL18}
V. Isakov and S. Lu, Increasing stability in the inverse source problem with
attenuation and many frequencies, SIAM J. Appl. Math., 78 (2018), 1--18. 

\bibitem{LPS08}
\textsc{M. Lassas, L. P\"aiv\"arinta, and E. Saksman}, {\em Inverse
scattering problem for a two dimensional random potential,} Comm. Math. Phys.,
279 (2008), pp. 669--703.

\bibitem{LHL}
\textsc{J. Li, T. Helin, and P. Li}, {\em Inverse random source problems for
time-harmonic acoustic and elastic waves,} arXiv:1811.12478.

\bibitem{LL19} 
\textsc{J. Li and P. Li}, {\em Inverse elastic scattering for a random
source,} SIAM J. Math. Anal., 51 (2019), pp. 4570--4603.  

\bibitem{LLM2}
\textsc{J. Li, H. Liu, and S. Ma}, {\em Determining a random Schr\"odinger
operator: both potential and source are random,} arXiv:1906.01240.

\bibitem{LLM1}
\textsc{J. Li, H. Liu, and S. Ma}, {\em Determining a random Schr\"odinger
equation with unknown source and potential,} arXiv:1811.00880.

\bibitem{LSSW16}
\textsc{A. Lodhia, S. Sheffield, X. Sun, and S. Watson}, {\em
Fractional Gaussian fields: a survey,} Probab. Surv., 13 (2016), pp. 1--56.

\bibitem{S70}
\textsc{E. M. Stein}, {\em Singular Integrals and Differentiability
Properties of Functions,} Princeton University Press, Princeton, N.J., 1970.


\end{thebibliography}
\end{document}